\documentclass[a4paper]{article}
\usepackage{latexstyle}
\usepackage[pagewise]{lineno}

\begin{document}


\title{The Vlasov-Navier-Stokes equations as a mean field limit}

\author{Franco Flandoli\footnote{franco.flandoli@sns.it. Scuola Normale Superiore of Pisa, Italy.}, Marta Leocata\footnote{leocata@mail.dm.unipi.it. University of Pisa, Italy.}, Cristiano Ricci\footnote{cristiano.ricci@unifi.it. University of Florence, Italy.}}
\maketitle

\begin{abstract}
Convergence of particle systems to the Vlasov-Navier-Stokes equations is a
difficult topic with only fragmentary results. Under a suitable modification
of the classical Stokes drag force interaction, here a partial result in this
direction is proven. A particle system is introduced, its interaction with the
fluid is modelled and tightness is proved, in a suitable topology, for the
family of laws of the pair composed by solution of Navier-Stokes equations and
empirical measure of the particles. Moreover, it is proved that every limit
law is supported on weak solutions of the Vlasov-Navier-Stokes system. Open
problems, like weak-strong uniqueness for this system and its relevance for
the convergence of the particle system, are outlined. 
\end{abstract}

\section{Introduction}
\let\thefootnote\relax\footnotetext{ \textbf{MSC 2010}:  35J05; 76T20; 60K35}

The Vlasov-Navier-Stokes (VNS) system is a mathematical model for a large
family of particles immersed in a fluid, the two objects interacting in a
simple but relevant way. We consider the following generalized form of VNS
system%
\begin{equation}\label{eq::NS-FP}
\left\{
\begin{aligned}
\frac{\partial u}{\partial t}&=\Delta u-u\cdot \nabla u-\nabla \pi- \int_{\mathbb{R}^d}g\left(u(t,x),v\right)F(x,v)\,dv \quad \div (u) = 0,\\
\frac{\partial F}{\partial t}&+v\cdot \nabla_x F+\div_v(g(u,v)F)=\frac{\sigma^2}{2}\Delta_v F,  \\
 u(0&,\cdot) = u_0,\quad F(0,\cdot,\cdot)= F_0,
\end{aligned}\right.
\end{equation}
where $u$ is the velocity of a 2D fluid, that for sake of simplicity we assume
to live in the torus $\Pi^{2}=\mathbb{R}^{2}/\mathbb{Z}^{2}$, $\pi$ is
the fluid pressure, $F$ is the density of particles immersed in the fluid, the
kinematic viscosity of the fluid is settled to one for simplicity of notations and the
diffusion coefficient of particles is $\sigma$.

The key issue of Vlasov-Navier-Stokes system is the interaction between
particles and fluid, represented by the forcing term $-\int_{\mathbb{R}^{d}%
}g\left(  u(t,x),v\right)  F(x,v)\,dv$ in the Navier-Stokes equations and by
the transport term $\operatorname{div}(g(u,v)F)$ in Vlasov equation. The
classical interaction is the so called Stokes drag force, modelled by
\[
g\left(  u,v\right)  =u-v
\]
up to constants. In this case the PDE system has been uniquely solved in
dimension 2, in a class of relatively smooth solutions, see \cite{Yu} and
references therein. For technical reasons we are not able to incude this case
in our main result on particle approximations and thus we consider a general
bounded function $g$ satisfying certain assumptions. This is a limitation that
we hope to overcome in future research. Boundedness of $g$ may be vaguely
justified on a physical ground by invoking a saturation effect of the
interaction when the difference in velocity between particles and fluid is too
large. In fact, a posteriori, the VNS system has a unique bounded solution
$\left(  u,F\right)  $, as proved by \cite{Yu}, and thus $g$ can be
elimiminated by the PDE system;\ but we cannot eliminate it from the particle
system below and at present we do not still know that the solutions of the
PDE\ system obtained as a limit of particles are bounded. Hence we have to
maintain $g$ in our presentation.

The shape of the PDE system clearly reminds a mean field model but, in spite
of several investigations, it seems very difficult to prove a full theorem.
Various particle models have been proposed, with several limitations; see for
instance \cite{Allaire}, \cite{Bernard}, \cite{Desvill}, \cite{Jabin} (see
also \cite{Desv PDE}, \cite{Goudon1}, \cite{Goudon2} for convergence of PDE to
PDE results, although motivated by particle arguments). Among the difficulties
discussed in these works there is the precise description of the interaction
between particles and fluid, which is a difficult topic in itself, see for
instance \cite{Bernard}, \cite{Dejard}, \cite{Fer}, \cite{gerard},
\cite{Sueur}. 

The present work is a continuation of \cite{Fla}; in the spirit of that work,
with the modification $g\left(  u,v\right)  $ introduced here, the particle
system has the form%

\begin{equation*}
\frac{\partial u^N}{\partial t}=\Delta u^N-u^N\cdot \nabla u^N-\nabla \pi^N-\frac{1}{N}\sum_{i=1}^{N} g\left(u^{N}_{\eps_{N}}(t,X^{i,N}_{t}),V^{i}_{t}\right)\delta^{\eps_{N}}_{X^{i,N}_{t}}
\end{equation*}

\begin{equation*}
\begin{cases}
dX^{i,N}_{t}=V^{i,N}_{t}dt\\
d V^{{i,N}}_{t}= g\left(u^{N}_{\eps_{N}}\left(t,X^{i,N}_{t}\right),V^{i,N}_{t}\right)dt+\sigma dW^{i}_{t}
\end{cases}
\end{equation*}

where $N$ is the number of particles, $X_{t}^{i,N}$ and $V_{t}^{i,N}$ are
position and velocity of particle $i=1,...,N$. We emphasize that, in the case treated here, it is necessary for technical reasons to introduce additive noise to system. 
In \cite{Fla}, a two-step approach is discussed in some detail: the parameters
$N$ = number of particles and $\epsilon$ = smoothing parameter of the
interaction (playing the role in a sense of the size of interaction of the
particle)\ are decoupled; first the limit as $N\rightarrow\infty$ is taken,
leading to a non-local PDE system; then the limit as $\epsilon\rightarrow0$ is
taken. Although being a first result, this is not the right way to understand
the link between the particle system and the PDE system. 

In the present work we couple the parameters $N$ and $\epsilon$ by taking a
sequence $\epsilon_{N}\rightarrow0$ as $N\rightarrow\infty$. The joint limit
introduces great difficulties. We present a way to overcome some of them under
the above mentioned additional assumptions on the interaction function $g$. We
prove that the particle system has subsequences which converge in law to
solutions of the PDE system. In order to prove that the full sequence
converges we should be able to prove uniqueness of solutions to the PDE system
in the class of solutions obtained as a limit of particles, but this class is
too weak a priori and we still do not know whether uniqueness holds. For
classes of more regular solutions, in particular with bounded density $F$,
uniqueness is known, see \cite{Yu}; and such solutions exists. Hence\ a
weak-strong uniqueness theorem would be sufficient, proving that our weaker
solutions coincide with the stronger ones provided by \cite{Yu}. We hope to
overcome these technical problems in future research.

\maketitle 

\section{Notations and Main Results}\label{sec::intro}
We introduce the microscopic  description of the PDE system \eqref{eq::NS-FP}. Particle system is described by the following dynamic
\begin{equation}\label{eq::particles}
\begin{cases}
dX^{i,N}_{t}=V^{i,N}_{t}dt\\
d V^{{i,N}}_{t}= g\left(u^{N}_{\eps_{N}}\left(t,X^{i,N}_{t}\right),V^{i,N}_{t}\right)dt+\sigma dW^{i}_{t}
\end{cases}
\end{equation}
\noindent with initial conditions $\left(X^{i,N}_0,V^{i,N}_0\right)\sim F_0\,dx\,dv$ i.i.d and $\mathcal{F}_0$-measurable, where $W^i$ is a sequence of independent Brownian motions on a filtered space $(\Omega, \mathcal{F},\mathcal{F}_t,\PP)$, and $\sigma\in \RR$ is the diffusion coefficient. Velocity of each particle is driven by the saturated stokes-drag force $g\in {C}^{1}_{b}(\Pi^{d}\times \RR^{d};\RR^{d})$. In our microscopic description the velocity interaction is computed between the particle velocity $V^i_t$ and a local average at particle center $X^i_t$ of the fluid velocity $u^N_{\eps_N}$:
\[
u^{N}_{\eps_{N}}(t,x)=(\theta^{0,\eps_{N}}*u^{N}_t)(x)
\]
where $\theta^{0,\eps_{N}}$ is a suquence of rescaled mollifier, described at the end of this section in \eqref{prop::theta0}-\eqref{eq:thetarescaled}.
At the microscopic level the fluid dynamic is given by the Naviers-Stokes equation, but forced by the presence of a discrete number of particles $N$.
\begin{equation}\label{eq::NSdiscrete}
\frac{\partial u^N}{\partial t}=\Delta u^N-u^N\cdot \nabla u^N-\nabla \pi^N-\frac{1}{N}\sum_{i=1}^{N} g\left(u^{N}_{\eps_{N}}(t,X^{i,N}_{t}),V^{i}_{t}\right)\delta^{\eps_{N}}_{X^{i,N}_{t}}
\end{equation}
with initial conditions $u^N_0$ such that $u^N_0\to u_0$ in $L^1(\Omega,L^2_w(\Pi^d))$. Again we adopt a sort of phenomenological
description, where particles act on the fluid as Dirac forces, with intensity dependent on the relative velocity between fluid and particle. For technical reasons, but also as a trace of the fact that particles occupy a volume, we use a smoothed version of Dirac to describe the force
$$\delta^{\eps_{N}}_{X^{i,N}_{t}}=\theta^{0,\eps_{N}}\left(x-X^{i,N}_{t}\right).$$

\noindent Setting 
\[
\mathcal{E}_t^{N}=\frac{1}{2}\int_{\Pi^{d}}\abs{u^{N}\left(t,x\right)}^{2}dx+\frac{1}{2N}\sum_{i=1}^{N}\abs{V^{i,N}_{t}}^{2}
\]
the total energy of the system and by applying It\^o formula, if $u^{N}$ is a regular solution then we have:

\begin{multline*}
d\mathcal{E}_{t}^{N}+  \left(\frac{1}{N}\sum_{i=1}^{N}g\left(  u^{N}_{\eps_{N}}\left(X^{i,N}_{t}\right), V^{i,N}_{t} \right)\left(u^{N}_{\eps_{N}}\left(X^{i,N}_{t}\right)- V^{i,N}_{t}\right)\right. \\+ \left. \int_{\Pi^{d}}\abs{\nabla u^{N}\left(t,x\right)}^{2}dx  \right)\, dt= 
\frac{d \cdot \sigma^{2}}{2}\,dt+\frac{\sigma}{N}\sum_{i=1}^{N}V^{i,N}_{t}dW^{i}_{t}.
\end{multline*}

The last identity is a consequence of the particular choice of the interaction between particles and the fluid velocity field.  Furthermore, if $g$ is such that $g(x,v)(x-v)\geq 0$ we obtain the traditional energy balance proper of the Vlasov-Navier-Stokes framework.

\begin{oss}\label{lem::energy}
In the classical case when $g\left(  u,v\right)
=u-v$, the term with 
$$\left\vert u_{\varepsilon_{N}}^{N}\left(  X_{t}
^{i,N}\right)  -V_{t}^{i,N}\right\vert ^{2}$$ on the right-hand-side has the
meaning of dissipation of energy due to Stokes drag force. Here, with the
function $g\left(  u,v\right)  $, it is natural to assume the condition
$g\left(  u,v\right)  \left(  u-v\right)  \geq0$ to preserve the
interpretation of a dissipation. However, thanks to the boundedness of $g$, we
do not need this condition in the mathematical analysis. Nonetheless we preferred to specify the hypothesis that guarantee the energy balance in order to keep the physical coherence of the system.
\end{oss}

For technical reasons (see Lemma \ref{lem:FNL4}), the mollifiers are chosen as 
$$\theta\left(x,v\right)=\theta^{0}\left(x\right)\theta^{1}\left(v\right)$$
with the following properties
	\begin{equation}\label{prop::theta0}
	\abs{\nabla\theta^{0}(x) }\leq \theta^{0}(x);
	\end{equation}

	\begin{equation}\label{prop::theta1}
	\text{supp}(\theta^{1}(v))\subseteq B(0,1).
	\end{equation}
	
\noindent and with the following rescaling:
\begin{equation}
\label{eq:thetarescaled}
\theta^{\eps_{N}}\left(x,v\right)=\eps_{N}^{-d}\theta\left(\eps_{N}^{-1}\left(x,v\right))\right)\quad \textrm{with}\quad \eps_{N}=N^{-\frac{\beta}{d}}\quad \beta<\frac{d}{3d+2}
\end{equation}
Denote by $\PP_{r_{1}}(\Pi^{d}\times\RR^{d})$ the set of probability measure on $\Pi^{d}\times \RR^{d}$ with finite first order moment. We also denote by $\mathcal{W}_{1}$ the usual Wasserstein-$1$ distance on such space. 
In this manuscript we denote with $S^N_t$ the empirical measure:
\[S^N_t=\frac{1}{N}\sum_{i=1}^N \delta_{\left(X^{i,N}_t,V^{i,N}_t\right)}.\]
With this notation we have that the measure $S^{N}_{t}$ belongs to $L^{1}(\Omega; C([0,T];\PP_{r_{1}}(\Pi^{d}\times \RR^{d})))$.
We are now able to state our main result.



\begin{teo}\label{teo::main}
Assuming that  $S^N_0$ converges to $F_0$ weakly in $L^1(\Omega)$ and $u^N_0$ converges to $u_0$ weakly in $L^1(\Omega)$, the family of laws $\{Q^N\}_{N \in \NN}$  of the couple $(u^{N}, S^{N})_{N \in \NN}$ is tight on $L^{2}([0,T];C(\Pi^{d}))\times C([0,T];\PP_{r_{1}}(\Pi^{d}\times\RR^{d}))$. 
Moreover any measure $Q$ limit of a subsequence $\{Q^{N_{k}}\}$ of $Q^{N}$, is supported on the weak solutions of system \eqref{eq::NS-FP} (see Definition \ref{defi::weaksolution}).

\end{teo}
A straightforward consequence of the previous result is the existence of weak solutions to \eqref{eq::NS-FP}, of class $L^{2}([0,T];C(\Pi^{d}))\times C([0,T];\PP_{r_{1}}(\Pi^{d}\times\RR^{d}))$.

\section{A priori Estimates}
\noindent 
In the following we will always assume $d=2,3$. In the next lemma, we state the identity satisfied by the empirical measure.
\begin{lem}
For every test function $\phi : \Pi^{d}\times \RR^{d} \to \RR$, $S^{N}_{t}$ satisfies the following identity:
\begin{eqnarray}\label{eq::empiricalmeasure}
d\langle S^{N}_{t},\phi \rangle = \langle  S^{N}_{t},v\cdot \nabla_{x}\phi\rangle\,dt +  \langle S^{N}_{t}, g(u^{N}_{\eps_{N}}(t,x),v)\cdot \nabla_{v}\phi  \rangle\,dt  \nonumber \\
+\frac{\sigma^{2}}{2}\langle S^{N}_{t}, \Delta_{v}\phi\rangle\,dt
+dM^{N,\phi}_{t}
\end{eqnarray}
with $M^{N,\phi}_{t}=\frac{\sigma}{N}\sum_{i=1}^{N}\int_{0}^{t}\nabla_{v}\phi\left(X^{i,N}_{t},V^{i,N}_{t}\right)\cdot dW^{i}_{t}$. Moreover $F^{N}_{t}=\theta^{\eps_{N}}*S^{N}_{t}$ satisfies:
\begin{eqnarray}\label{eq::empiricalmeasuremollified}
dF^{N}_{t}=\frac{\sigma^{2}}{2}\Delta_{v}F^{N}_{t}-\div_{v}(\theta^{\eps_{N}}*(g(u^{N}_{\eps_{N}}(t,x),v)S^{N}_{t}))\,dt\nonumber\\ 
- \div_{x}(\theta^{\eps_{N}}*(vS^{N}_{t}))\,dt+dM_{t}^{N,\eps_{N}}\end{eqnarray}
with $M^{N,\eps_{N}}_{t}=M^{N,\theta^{\eps_{N}}\left(x-\cdot,v-\cdot\right)}_{t}$.

\end{lem}
\begin{proof}
The proof of the identity is classical, it simply follows from an application of It\^{o} formula. The second identity follows from the first applying the first $\left(\ref{eq::empiricalmeasure}\right)$ with $\phi=\theta^{\eps_{N}}\left(x-\cdot,v-\cdot \right)$
\end{proof}

\noindent
To have a first overview of the main difficulties of the problem,  we start from the identity satisfied by the empirical measure: the most troublesome term is the nonlinear one
\[ \langle S^{N}_{t}, g(u^{N}_{\eps_{N}}(x),v)\cdot \nabla_{v}\phi \rangle.\]
We need to prove that the following limit holds

\[\lim_{N\to\infty} \langle S^{N}_{t}, g(u^{N}_{\eps_{N}}(x),v)\cdot \nabla_{v}\phi \rangle= \langle F_t, g(u(t),v) \cdot \nabla_{v}\phi \rangle. \]
Assuming that $S^N_t$ converges to $F_t$ weakly as probability measure, to prove the identity above, one can prove that $u^N_t$ converges to $u$ uniformly. 
Thus, by mean of Sobolev embedding, to obtain compactness criteria in $C(\Pi^d)$ we look for an a priori estimate 
in $W^{2,2}(\Pi^d)$. To do so we approach $W^{1,2}(\Pi^d)$ estimate of the vorticity $\omega^N=\nabla^\perp\cdot u^N$. We recall that the vorticity $\omega^{N}$ satisfies the following PDE:
\begin{equation}\label{eq::omegaNvorticity}
\partial_{t}\omega^{N} = \Delta \omega^{N}- u^{N}\cdot \nabla \omega^{N} - \frac{1}{N}\sum_{i=1}^{N} g\left(u^{N}_{\eps_{N}}(t,X^{i,N}_{t}),V^{i}_{t}\right)\nabla^{\perp}\cdot \delta^{\eps_{N}}_{X^{i,N}_{t}} 
\end{equation}
\subsection{Main estimates on the marginal distribution $F^{N,0}$}
Thanks to the saturation of the Stokes drag force, the proof of tightness of the laws of empirical measures $S^N_t$ is classical. The main difficult is contained in the proof of tightness of the velocity field $u^N$. Again, thanks to the saturation of the Stokes drag force, to estimate the term of interaction between particles and fluid, we just need to get estimate on the marginal distribution $F^{N,0}$ in $L^{2}$ norm. Inspired by
\cite{Boud}, we prove in lemma \ref{lem::margmomentsfinal} an inequality between the $L^{2}$ norm of the marginal and the $L^{4}$ norm of the full distribution, together with some moments on the $v$  component. 
This result is a variant of Lemma 1 of \cite{Boud}, which avoids $||F^N_t||_\infty$, since it is too difficult to estimate.
\begin{lem}
\label{lem::margmomentsfinal}
Calling $F^{N,0}_t(x)$ the marginal of the density $F^N_t(x,v)$, i.e.
\[
F^{N,0}_t(x) = \int_{\RR^d} F^N_t(x,v)\,dv,
\]
The following holds
\begin{equation}\label{ineq:marginal_total}
\int_{\Pi^{d}}F^{N,0}_{t}(x)^{2}\,dx \leq K\int_{\RR^{d}}\int_{\Pi^{d}} \abs{v}^{\frac{3d}{2}}F^{N}_{t}(x,v)\,dxdv + K'\int_{\RR^{d}}\int_{\Pi^{d}}F^{N}_{t}(x,v)^{4}\,dxdv.
\end{equation}

\end{lem}
\begin{proof}
We write for a.e. $(t,x)$:
\begin{multline*}
F_{t}^{N,0}\left( x\right)  =\int_{\left\vert v\right\vert \leq r\left(
x\right) }F_{t}^{N}\left( x,v\right) dv+\int_{\left\vert v\right\vert
>r\left( x\right) }F_{t}^{N}\left( x,v\right) dv \\
\leq C\left( \int_{\left\vert v\right\vert \leq r\left( x\right)
}F_{t}^{N}\left( x,v\right) ^{4}dv\right) ^{1/4}r\left( x\right)^{\frac{d3}{4}} +\\
\frac{1}{r\left( x\right) ^{\alpha }}\int_{\left\vert v\right\vert >r\left(
x\right) }\left\vert v\right\vert ^{\alpha }F_{t}^{N}\left( x,v\right) dv
\end{multline*}%
Passing to the $L^2$-norm and imposing that $r\left( x\right)=\left( \int_{\left\vert v\right\vert >r\left( x\right)
}\left\vert v\right\vert ^{\alpha }F_{t}^{N}\left( x,v\right) dv\right)^{1/2\alpha}$
\begin{multline*}
\int F_{t}^{N,0}\left( x\right) ^{2}dx \leq \\
\leq2C\int \left( \int_{\left\vert
v\right\vert >r\left( x\right) }\left\vert v\right\vert ^{\alpha
}F_{t}^{N}\left( x,v\right) dv\right) ^{\frac{d3}{4\alpha }}\left(
\int_{\left\vert v\right\vert \leq r\left( x\right) }F_{t}^{N}\left(
x,v\right) ^{4}dv\right) ^{1/2}dx \\
+2\int \int_{\left\vert v\right\vert >r\left( x\right) }\left\vert
v\right\vert ^{\alpha }F_{t}^{N}\left( x,v\right) dxdv.
\end{multline*}%
Letting $\alpha = \frac{d3}{2}$, and by Young inequality, we get the desired result \eqref{ineq:marginal_total}.
\end{proof}
To get the $L^2$-estimate on the marginal distribution we start by $L^4$-estimate of the density $F^N$.
\begin{lem}
\label{lem:FNL4}
If $\theta^{0}(v)$ and $\theta^{1}(v)$ satisfies conditions (\ref{prop::theta0}) and (\ref{prop::theta1}), then there exists a constant $C$, independent on $N$, such that
\[
\sup_{t\in[0,T]}\EE{\int_{\RR^{d}}\int_{\Pi^{d}}\abs{F^{N}_{t}(x,v)}^{4}\,dxdv} \leq C.
\]
\end{lem}
\begin{proof}
By It\^o formula and integration by parts we have 
\begin{equation}\label{eq:lhsFN4}
\frac{1}{4}d\int_{\RR^{d}}\int_{\Pi^{d}} (F^{N}_{t})^{4}\,dxdv + \frac{3\sigma^{2}}{2}\int_{\RR^{d}}\int_{\Pi^{d}}(F^{N}_{t})^{2}\abs{\nabla_{v}F^{N}_{t}}^{2}\,dxdvdt =	
\end{equation}
\begin{equation}
\label{eq:FN4,divx}
-\int_{\RR^{d}}\int_{\Pi^{d}}(F^{N}_{t})^{3}\div_{x}(\theta^{\eps_{n}}*(vS^{N}_{t}))\,dxdvdt
\end{equation}
\begin{equation}
\label{eq:FN4,divv}
- \int_{\RR^{d}}\int_{\Pi^{d}}(F^{N}_{t})^{3}\div_{v}(\theta^{\eps_{n}}*\left(g(u^{N}_{\eps_{N}},v)S^{N}_{t}\right))\,dxdvdt
\end{equation}
\begin{equation}
\label{eq:FN4,reminder}
+ \int_{\RR^{d}}\int_{\Pi^{d}}(F^{N}_{t})^{3}\,dM^{N,\eps_{N}}_{t}\,dxdv + 3\int_{\RR^{d}}\int_{\Pi^{d}}(F^{N}_{t})^{2}\,d[M^{N,\eps_{N}}]_{t}\,dxdv.
\end{equation}
Let us deal with the terms above separately. We can rewrite the inner term in \eqref{eq:FN4,divx} as 
\[
\div_{x}(\theta^{\eps_{n}}*(vS^{N}_{t})) =  v\cdot \nabla_{x}(\theta^{\eps_{N}}*S^{N}_{t}) - ((\nabla_{x}\theta^{\eps_{N}}\cdot v)*S^{N}_{t}).
\]
The first term on the r.h.s. leads to 
\[
\int_{\RR^{d}}\int_{\Pi^{d}}(F^{N}_{t})^{3}\nabla_{x}F^{N}_{t}\cdot v\,dxdv = \frac{1}{4}\int_{\RR^{d}}\int_{\Pi^{d}}\nabla_{x}(F^{N}_{t})^{4}\cdot v\,dxdv = 0
\]
while the second, being a little bit trickier is argued below. Concerning the integrand inside the convolution, 
 by using the assumption on $\theta^{0}(x)$ and $\theta^{1}(v)$ stated conditions (\ref{prop::theta0}) and (\ref{prop::theta1}), it can be treated with the following argument
\[
 \abs{\nabla_{x}\theta^{0,\eps_{N}}(x-x')}\theta^{1,\eps_{N}}(v-v')\abs{(v-v')}
\]
\[
= \eps_{N}^{-1}\eps_{N}^{-d}\abs{\nabla_{x}\theta^{0}(\eps_{N}^{-1}(x-x'))}\eps_{N}^{-d}\theta^{1}(\eps_{N}^{-1}(v-v'))\abs{v-v'}
\]
\[
\leq \eps_{N}^{-d}\theta^{0}(\eps_{N}^{-1}(x-x'))\eps_{N}^{-d}\theta^{1}(\eps_{N}^{-1}(v-v'))\frac{\abs{v-v'}}{\eps_{N}}
\]
\[
\leq \theta^{0,\eps_{N}}(x-x')\theta^{1,\eps_{N}}(v-v')2
\]
leading to
\[
\abs{((\nabla_{x}\theta^{\eps_{N}}\cdot v)*S^{N}_{t})} \leq 2F^{N}_{t}
\]
implying 
\[
\abs{\int_{\RR^{d}}\int_{\Pi^{d}}(F^{N}_{t})^{3}\div_{x}(\theta^{\eps_{n}}*(vS^{N}_{t}))\,dxdv} \leq 2\norm{F^{N}_{t}}_{L^{4}}^{4}
\]
which ends $\eqref{eq:FN4,divx}$. 

Concerning \eqref{eq:FN4,divv}, it follows easily by the boundedness of $g$ that
\[
\abs{\div_{v}(\theta^{\eps_{N}}*g(u^{N}_{\eps_{N}},v)S^{N}_{t})}\leq \norm{g}_{L^{\infty}}\abs{\nabla_{v}F^{N}_{t}},
\]
then
\begin{multline*}
\abs{\int_{\RR^{d}}\int_{\Pi^{d}}(F^{N}_{t})^{3}\div_{v}(\theta^{\eps_{n}}*g(u^{N}_{\eps_{N}},v)S^{N}_{t})\,dxdv} \leq \\
\leq {C_{g}}\int_{\RR^{d}}\int_{\Pi^{d}}(F^{N}_{t})^{2}\abs{F^{N}_{t}\nabla_{v}F^{N}_{t}}\,dx\,dv
\\
\leq \frac{C_{g}^{2}}{\delta}\norm{F^{N}_{t}}_{L^{4}}^{4} + \delta\int_{\RR^{d}}\int_{\Pi^{d}}(F^{N}_{t})^{2}\abs{\nabla_{v}F^{N}_{t}}^{2}\,dxdv
\end{multline*}
and by choosing $\delta$ small enough we can take the second term to the l.h.s. of \eqref{eq:lhsFN4} maintaining the positivity.

Let us now deal with \eqref{eq:FN4,reminder}: the integral with respect to the martingale will vanish when taking the average, while the second term is
\begin{multline}\label{eq:FN4,qv}
\int_{\RR^{d}}\int_{\Pi^{d}}(F^{N}_{t})^{2}\,d[M^{N,\eps_{N}}]_{t}\,dxdv = \frac{\sigma^{2}}{N}\int_{\RR^{d}}\int_{\Pi^{d}} (F^{N}_{t})^{2}(\abs{\nabla_{v}\theta^{\eps_{N}}}^{2}*S^{N}_{t})\,dxdvdt\leq \\
\sigma^{4}\norm{F^{N}_{t}}_{L^{4}}^{4}dt + \frac{1}{N^{2}}\int_{\RR^{d}}\int_{\Pi^{d}}(\abs{\nabla_{v}\theta^{\eps_{N}}}^{2}*S^{N}_{t})^{2}\,dxdvdt.
\end{multline}
Notice that 
\begin{multline*}
\frac{1}{N^{2}}\int_{\RR^{d}}\int_{\Pi^{d}}(\abs{\nabla_{v}\theta^{\eps_{N}}}^{2}*S^{N}_{t})^{2}\,dxdv \leq \\ \leq \frac{1}{N^{2}}\frac{1}{N}\sum_{i=1}^{N}\int_{\RR^{d}}\abs{\nabla_{v}\theta^{1,\eps_{N}}(v-V^{i,N}_{t})}^{4}\,dv\int_{\Pi^{d}}\theta^{0,\eps_{N}}(x-X^{i,N}_{t})^{4}\,dx
\end{multline*}
and 
\[
\int_{\RR^{d}}\abs{\nabla_{v}\theta^{1,\eps_{N}}(v-V^{i,N}_{t})}^{4}\,dv = c N^{\frac{4\beta(d+1)}{d}-\beta}\]
\[ \int_{\Pi^{d}}\theta^{0,\eps_{N}}(x-X^{i,N}_{t})^{4}\,dx = c'N^{3\beta},
\]
by using the explicit form for the scaling factor $\eps_{N}$.
Substituting into \eqref{eq:FN4,qv} we get 
\[
\leq \frac{c''}{N^{3}}\sum_{i=1}^{N} N^{3\beta}\cdot N^{\frac{4\beta(d+1)}{d}-\beta} \,dt= c''N^{3\beta - \frac{4\beta(d+1)}{d}-\beta -2 }\,dt 
\]
which tend to zero for $\beta < \frac{d}{3d+2}$.

Summarizing we have obtained
\[
d\norm{F^{N}_{t}}_{L^{4}}^{4} + \left(  \frac{3\sigma^{2}}{2}-\delta\right)\int_{\RR^{d}}\int_{\Pi^{d}}(F^{N}_{t})^{2}\abs{\nabla_{v}F^{N}_{t}}^{2}\,dxdvdt \leq 
\]
\[
\leq K_{g}\norm{F^{N}_{t}}_{L^{4}}^{4}\,dt+ \int_{\RR^{d}}\int_{\Pi^{d}}(F^{N}_{t})^{3}\,dM^{N,\eps_{N}}_{t}\,dxdv + K' dt
\]
which, after taking the average, ends the proof by standard Gronwall lemma.
\end{proof}

\begin{lem}
\label{lem:FNmomentsV}
For all $\gamma > 0$ there exists a constant $C$, independent on $N$, such that
\begin{equation}
\label{eq:lemmomentsV}
\EE{ \sup_{t \in [0,T]} \int_{\RR^{d}}\int_{\Pi^{d}}\abs{v}^{\gamma}F^{N}_{t}(x,v)\,dxdv}\leq C.
\end{equation}
\end{lem}
\begin{proof}
The thesis follows by expanding $F^{N}_{t}$ as a summation and noticing that 
\[
\eps_{N}^{-d}\int_{\RR^{d}}\abs{v}^{\gamma}\theta^{1}(\eps_{N}^{-1}(v-V^{i,N}_{t}))\,dv \stackrel{v'=\eps_{N}^{-1}(v-V^{i,N}_{t})}{=} \int_{\RR^{d}}\abs{\eps_{N}v'+V^{i,N}_{t}}^{\gamma}\theta^{1}(v')\,dv'\leq
\]
\[
\eps_{N}^{\gamma}C+\abs{V^{i,N}_{t}}^{\gamma}C'.
\]
Substituting the previous in the l.h.s. of \eqref{eq:lemmomentsV} we obtain
\[
\EE{\sup_{t \in [0,T]}\int_{\RR^{d}}\int_{\Pi^{d}}\abs{v}^{\gamma}F^{N}_{t}(x,v)\,dxdv}\leq \eps_{N}^{\gamma}C+C'\EE{\sup_{t\in [0,T]}\abs{V^{i,N}_{t}}^{\gamma}}\leq C''_{T,\gamma}
\]
where the last inequality is due to 
\[
\EE{\sup_{t\in [0,T]} \abs{V^{i,N}_{t}}^{\gamma}}\leq C_{T,\gamma}
\]
which follows easily from the equation for $V^{i,N}_{t}$ \eqref{eq::particles} and from the boundedness of $g$.
\end{proof}


\begin{lem}
\label{lem:FN0L2}
There exists a constant $C$, independent on $N$, such that
\begin{equation}
\sup_{t\in[0,T]}\EE{\int_{\Pi^{d}}F^{N,0}_{t}(x)^{2}\,dx}\leq C.
\end{equation}
\end{lem}
\begin{proof}
Taking the expected value both sides of inequality stated in \eqref{ineq:marginal_total} and applying lemma \ref{lem:FNL4} and \ref{lem:FNmomentsV} we conclude the proof.	
\end{proof}
\subsection{Uniform estimates on the velocity field $u^{N}$}
As explained in the previous section, to study the tightness of the law of $u^N_t$ we will approach the vorticity $\omega^N_t = \nabla^{\perp} \cdot u^{N}$. First, to simplify the notation, we call
\[
G^{N}(x) = \frac{1}{N}\sum_{i=1}^{N}g(u^{N}_{\eps_{N}}(X^{i,N}_{t},V^{i,N}_{t})\delta^{\eps_{N}}_{X^{i,N}_{t}}(x),
\]
With this notation equation \eqref{eq::omegaNvorticity} becomes 
\[
\frac{\partial \omega^{N}}{\partial t} = \Delta \omega^{N} - u^{N}\cdot \nabla \omega^{N}  -\nabla^{\bot}\cdot G^{N}.
\]

\begin{lem}
\label{lem:omegaNenergy}
There exists a constant $C$, independent on $N$, such that
\begin{equation}
\EE{\norm{\omega^{N}}_{L^{2}([0,T];W^{1,2}(\Pi^{d}))}} \leq C
\end{equation}
\end{lem}
\begin{proof}

Notice that 
\begin{equation}\label{eq::G}
\abs{G^{N}(x)}\leq \norm{g}_{L^{\infty}}F^{N,0}_{t}(x).
\end{equation}
A classical energy approach, applied to the vorticity equation leads to
\[
\norm{\omega^{N}_{t}}_{L^{2}}^{2}+\int_{0}^{t}\norm{\nabla \omega^{N}_{t}}_{L^{2}}^{2}\,dt \leq \norm{\omega^{N}_{0}}_{L^{2}}^{2}+\abs{\int_{0}^{t}\int_{\Pi^{d}}\omega^{N}_{t}(x)\nabla^{\bot}\cdot G^{N}(x)\,dxdt}.
\]
Using the \eqref{eq::G} and integration by parts, we also have
\begin{multline*}
\abs{\int_{\Pi^{d}}\omega^{N}_{t}(x)\nabla^{\bot}\cdot G^{N}(x)\,dx}\leq \int_{\Pi^{d}}\abs{\nabla \omega^{N}_{t}(x)}F^{N,0}_{t}(x)\,dx\leq \\ \frac{1}{2}\norm{\nabla \omega^{N}_{t}}_{L^{2}}^{2}+2\norm{F^{N,0}_{t}}_{L^{2}}^{2}
\end{multline*}
which implies 
\[
\sup_{t\in [0,T]}\norm{\omega^{N}_{t}}_{L^{2}}^{2}+\frac{1}{2}\int_{0}^{T}\norm{\nabla \omega^{N}_{t}}_{L^{2}}^{2}\,dt \leq \norm{\omega^{N}_{0}}_{L^{2}}^{2}+2\int_{0}^{T}\norm{F^{N,0}_{t}}_{L^{2}}^{2}\,dxdt.
\]
The proof ends by taking the expected value both sides and using lemma \ref{lem:FN0L2}.
\end{proof}

\begin{lem}
\label{lem:omegaNW12}
For all $\eps$ there exists $R$ such that 
\[
\PP(\norm{\omega^{N}}_{W^{1,2}([0,T]);H^{-1}(\Pi^{d})} > R) \leq \eps
\]
\end{lem}
\begin{proof}
By lemma \ref{lem:omegaNenergy} we already know that $\omega^{N}$ is uniformly bounded in \\ $L^{2}([0,T];W^{1,2}(\Pi^{d}))$. We have
\[
\norm{\frac{\partial \omega^{N}}{\partial t}}_{H^{-1}}\leq \norm{\Delta \omega^{N}}_{H^{-1}} + \norm{u^{N}\cdot \nabla \omega^{N}}_{H^{-1}} + \norm{g}_{\infty} \norm{\nabla^{\bot}\cdot F^{N,0}}_{H^{-1}}
\]
\[
\leq \norm{\nabla \omega^{N}}_{L^{2}} + \norm{\omega^{N}}_{L^{2}}^{2}\norm{u^{N}}_{\infty}^{2}+\norm{F^{N,0}}_{L^{2}}
\]
by classical arguments and integration by parts. To obtain the desired result is enough to note that 
\begin{multline*}
\PP \left( \sup_{t\in[0,T]}\norm{\omega^{N}}_{L^{2}}^{2}\int_{0}^{T}\norm{u^{N}_{s}}_{\infty}^{2}\,ds > \frac{R}{3} \right) \leq \PP \left( \sup_{t\in[0,T]}\norm{\omega^{N}}_{L^{2}}^{2} > \sqrt{\frac{R}{3}} \right)+\\ \PP \left( \int_{0}^{T}\norm{u^{N}_{s}}_{\infty}^{2}\,ds > \sqrt{\frac{R}{3}} \right) < 2\eps
\end{multline*}
by Sobolev embedding of $W^{2,2}(\Pi^{2})$ into $C(\Pi^{2})$, lemma \ref{lem:FN0L2} and \ref{lem:omegaNenergy}.\end{proof}

\section{Tightness of $(u^{N},S^{N})$ and Passage to the limit}

\begin{lem}
\label{lem:PretightSN}
For all $p > 0$ there exist $\beta,C> 0$, independently on $N$, such that the following conditions hold
\[
\EE{\sup_{t\in[0,T]}\int_{\RR^{d}}\int_{\Pi^{d}}\left(\abs{x}+\abs{v}\right)S^{N}_{t}(dx,dv)}\leq C,
\]
\[
\EE{\mathcal{W}_{1}(S^{N}_{t},S^{N}_{s})^{p} }\leq C\abs{t-s}^{1+\beta}.
\]
\end{lem}
\begin{proof}
The result follows easily from the boundedness of the coefficient in equation \eqref{eq::particles}. 
\end{proof}
\begin{teo}
\label{teo::tightSN}
Under assumption of section \ref{sec::intro} the family of laws $\{Q^{N,S}\}_{N\in\NN}$ of the empirical measure $\{S^{N}_{\cdot}\}_{N\in\NN}$ is relatively compact with respect of the weak convergence on $C\left([0,T];\PP_{{r}_{1}}(\Pi^{d}\times \RR^{d})\right)$.
\end{teo}
\begin{proof}
From the estimates in lemma \ref{lem:PretightSN}, with relatively classical compactness theorems, one can prove the tightness of the family of $\{Q^{N,S}\}_{N\in\NN}$. Relative compactness of measures then follows by Prohorov Theorem.
\end{proof}

\begin{teo}\label{teo::tightuN}
Under assumption of section \ref{sec::intro} the family of laws $\{Q^{N,u}\}_{N\in\NN}$ of $\{u^{N}\}_{N\in\NN}$, is relatively compact as a probability measure on $L^{2}([0,T];C(\Pi^{d}))$.
\end{teo}
\begin{proof}
For $M,R > 0 $ introduce
\begin{multline*}
K_{M,R} = \{ u \in L^{2}([0,T];C(\Pi^{d})) \,|\, \norm{u}_{L^{2}([0,T];W^{2,2}(\Pi^{d}))} \leq M, \\ \norm{u}_{W^{1,2}([0,T];L^{2}(\Pi^{d}))} \leq R \}
\end{multline*}
which by the Aubin-Lions Lemma is relatively compact in $L^{2}([0,T];C(\Pi^{d}))$. Notice that
\begin{multline*}
Q^{N,u}(K^{c}_{M,R}) = \PP(u^{N}\in K^{c}_{M,R}) 
\leq \\ \PP(\norm{u^{N}}_{L^{2}([0,T];W^{2,2}(\Pi^{d}))} > M)+
\PP(   \norm{u^{N}}_{W^{1,2}([0,T];L^{2}(\Pi^{d}))} > R )\leq 
\\
\leq \frac{\EE {\norm{u^{N}}_{L^{2}([0,T];W^{2,2}(\Pi^{d}))}  }}{M} + \eps
\end{multline*}
by lemma \ref{lem:omegaNW12}. By lemma \ref{lem:omegaNenergy} the expected values on the r.h.s. is uniformly bounded with respect to $N$, hence the sequence $\{Q^{N,u}\}_{N\in\NN}$ is tight and proof is concluded. 

\end{proof}

\noindent As a consequence of theorems \ref{teo::tightuN} \ref{teo::tightSN}, we can conclude that from each subsequence of laws $Q^N$ we can extract a further subsequence $Q^{N_k}:= Q^{N_k,S} \times Q^{N_k,u}$ that converges to a measure $Q$ on $L^2([0,T];C(\Pi^d))\times C([0,T];\PP_{r_{1}}\left( \Pi^{d}\times\RR^{d}\right))$. In  theorem \ref{teo::limit}, we prove that the limit measure $Q$ is supported on the weak solutions, as defined below of the coupled system \eqref{eq::NS-FP}, but we can not conclude that 
the full sequence of laws $Q^{N}$ converges to the solutions of the PDE system, because we are not able to prove uniqueness in the class of weak solutions.

\begin{defi}\label{defi::weaksolution}
A weak solution of the system composed by Navier-Stokes equation and the Fokker-Planck equation is a couple $(u,F)$ where $u \in L^2([0,T];C(\Pi^d))$ and $F \in C([0,T];\PP_{r_1}(\Pi^d\times \RR^d))$ such that for each divergence free smooth fields $\phi:[0,T]\times\Pi\to\RR^{d}$, and for each smooth compact support function
$\psi:[0,T]\times \Pi^d\times \RR^{d}\to\RR$
\begin{eqnarray*}
\langle u_t,\phi_t\rangle=\langle u_0,\phi_0\rangle+\int_0^t \langle u_s,\frac{\partial\phi_s}{\partial s}\rangle ds +\int_0^t\langle u_s,\Delta \phi_s\rangle ds+\int_0^t \langle u_s\cdot\nabla \phi_s,u_s\rangle ds\\
+\int_0^t\langle\pi_s, \nabla \phi_s\rangle ds
-\int_0^t\int_{\Pi^d}\int_{\mathbb{R}^d}\phi_s(x)g(u_s(x),v)F_s(dx,dv)ds, 
\end{eqnarray*}
\begin{eqnarray}\label{eq::weakNS-FP}
\langle F_t,\psi_t\rangle=\langle F_0,\psi_0\rangle+\int_{0}^{t}\langle F_{s},\frac{\partial \psi_{s}}{\partial s} \rangle ds +\int_0^t \langle F_s,\Delta_v \psi_s\rangle ds+\nonumber\\ \int_0^t\langle F_s,v\cdot \nabla_x \psi_s\rangle ds+ 
\int_0^t \langle F_s,g(u_s,v)\cdot \nabla _v \psi_s\rangle ds.\nonumber \\
\end{eqnarray}
\end{defi}

\begin{teo}\label{teo::limit}
Under the assumption of section \ref{sec::intro}, Q is supported on the weak solutions $(\ref{eq::weakNS-FP})$ of the system \eqref{eq::NS-FP}
\end{teo}
\begin{proof} 
\noindent {\bf Step 1.} We introduce two functionals $\Phi_\phi, \Psi_\psi$, dictated by the weak solutions  $(\ref{eq::weakNS-FP})$. We denote with $H$ the space $L^2([0,T];C(\Pi^d))\times C([0,T];\PP_{r_{1}}\left( \Pi^{d}\times\RR^{d}\right))$. The first $\Phi_\phi$ is defined for each free divergence smooth field $\phi$:
\[\Phi_\phi:H\to\RR\]
\begin{multline*}
\Phi_\phi(u,F)=\langle u_T,\phi_T\rangle-\langle u_0,\phi_0\rangle-\int_0^T \langle u_s,\frac{\partial\phi_s}{\partial s}\rangle ds -\int_0^T\langle u_s,\Delta \phi_s\rangle ds\\-\int_0^T \langle u_s\cdot\nabla \phi_s,u_s\rangle ds
-\int_0^T\langle  \pi_s,\nabla\phi_s\rangle ds
+\int_0^T\int_{\Pi^d}\int_{\mathbb{R}^d}\phi_s(x)g(u_s(x),v)F_s(dx,dv)ds\, .\\
\end{multline*}
The second functional $\Psi_\psi$ is defined for each smooth function $\psi$:
\[\Psi_\psi:H\to\RR\]
\begin{multline*}
\Psi_\psi(u,F)=\langle F_T,\psi_T\rangle-\langle F_0,\psi_0\rangle-\int_{0}^{t}\langle F_{s},\frac{\partial \psi_{s}}{\partial s} \rangle ds-\int_0^T\langle F_s,\Delta_v \psi_s\rangle ds\\-\int_0^T\langle F_s,v\cdot \nabla_x \psi\rangle ds 
-\int_0^T \langle F_s,g(u_s,v)\cdot \nabla _v \psi\rangle ds\,.
\end{multline*}
Notice that the functionals are continuous on 
$L^2([0,T];C(\Pi^d))\times C([0,T];\PP_{r_{1}}\left( \Pi^{d}\times\RR^{d}\right))$ endowed respectively with the strong topology and the weak topology of measures. Continuity of both functionals is easily checked due to uniform convergence. Now, we consider 
$Q^{N}$, the product measure, between the law of $u^{N}_{t}$ and $S^{N}_{t}$ $Q^{N,u}\otimes Q^{N,S}$. By theorems $(\ref{teo::tightSN})$ and $(\ref{teo::tightuN})$, there exists a subsequence $N_{k}$ such that $Q^{N_{k}}$ weakly converges to $Q$. By Portmanteau theorem, we get:

\begin{multline*}
Q\left((u,F)\in H : \abs{\Phi_\phi(u,F)}+\abs{\Psi_\psi(u,F)}> \delta  \right)\leq \\\liminf_{k\to\infty} Q^{N_{k}}\left((u,F)\in H : \abs{\Phi_\phi(u,F)}+\abs{\Psi_\psi(u,F)}>\delta\right)\\
=\liminf_{k\to\infty} \PP\left(\abs{\Phi_\phi(u^{N_{k}},S^{N_{k}})}+\abs{\Psi_\psi(u^{N_{k}},S^{N_{k}})}>\delta\right).
\end{multline*}


To prove that $Q\left((u,F)\in H : \abs{\Phi_\phi(u,F)}+\abs{\Psi_\psi(u,F)}>\delta\right)=0$, we prove in Step 2 that the r.h.s is zero for each $\delta >0$. 
By a classical density argument we can conclude that:
$$Q\left((u,F)\in H : \abs{\Phi_\phi(u,F)}+\abs{\Psi_\psi(u,F)}=0 \quad \forall \phi,\forall\psi\right)=1.$$ 

\noindent {\bf Step 2:} It remains to prove that 
\[\liminf_{k\to\infty} \PP\left(\abs{\Phi_\phi(u^{N_{k}},S^{N_{k}})}+\abs{\Psi_\psi(u^{N_{k}},S^{N_{k}})}>\delta\right)=0.\]
From Chebyschev inequality, we get that 
\[\PP\left(\abs{\Phi_\phi(u^{N_{k}},S^{N_{k}})}+\abs{\Psi_\psi(u^{N_{k}},S^{N_{k}})}>\delta\right)\leq \frac{\EE{  \abs{\Phi_\phi(u^{N_{k}},S^{N_{k}})}+\abs{\Psi_\psi(u^{N_{k}},S^{N_{k}})}}}{\delta}.\]
\noindent For simplicity we will write $N$ instead of $N_{k}$. Now, we focus on the first term of the sum. By the weak formulation of Navier-Stokes for the microscopic model we get:
\begin{eqnarray*}
\Phi_\phi(u^N,S^N)=\langle u^N_0-u_0,\phi_0\rangle+\int_0^T\int_{\Pi^d}\phi_s(x)\frac{1}{N}\sum_{i=1}^N g(u^N_{\eps_N}(X^{i,N}_s),V^{i,N}_s)\delta^{\eps_{N}}_{X^{i,N}_t}(x) dxds\\
 -\int_0^T\int_{\Pi^d}\int_{\mathbb{R}^d}\phi_s(x)g(u^N_s(x),v)S^N_s(dx,dv)ds.
\end{eqnarray*}
By some calculations, one can check that:
\begin{multline*}
\int_0^T\int_{\Pi^d}\phi_s(x)\frac{1}{N}\sum_{i=1}^N g(u^N_{\eps_N}(X^{i,N}_s),V^{i,N}_s)\delta^{\eps_{N}}_{X^{i,N}_t}(x) dxds=\\ \int_0^T\int_{\Pi^d}\int_{\mathbb{R}^d}g(u^N_{\eps_N}(x),v)\phi^{\eps_N}_s(x)S^N_{s}(dx,dv)ds.
\end{multline*}
Where $\phi^{\eps_{N}}=\theta^{\eps_{N}}*\phi$. Now, we can estimate the remaining term:
\begin{multline*}
\EE{ \left|\Phi_\phi(u^N,S^N)\right|}\leq \\
\EE{\left|\int_0^T\int_{\Pi^d}\int_{\mathbb{R}^d}\left(\phi^{\eps_N}_s(x)g(u^N_{\eps_N}(x),v))-\phi_s(x)g(u^N_s(x),v)\right)S^N_s(dx,dv)ds\right|}+\\
+ \EE{\langle u^N_0-u_0,\phi_0\rangle}.
\end{multline*}
For the hypothesis on initial condition, the second term is infinitesimal. Let's estimate the first summand.
\begin{multline*}
\mathbb{E}\bigg[\int_0^T\int_{\Pi^d}\int_{\mathbb{R}^d}\left|g(u^N_{\eps_N}(x),v))||\phi^{\eps_N}_s(x)-\phi_s(x)\right|+\\
 \left|\phi_s(x)||g(u^N_{\eps_N}(x),v)-g(u^N_s(x),v)\right|S^N_s(dx,dv)ds\bigg].
\end{multline*}

The first summand tends to zero, because of the boundness of $g$ and because $\phi^{\eps_N}\to \phi$ uniformly. Regarding the second summand, for the lipschtzianity of g, we have:
\[\abs{g(u^N_{\eps_N}(x),v)-g(u^N_s(x),v)}\leq C\abs{u^N_{\eps_N}(x)-u^N_s(x)}. \]
Notice that $\abs{u^N_{\eps_N}-u^N_s}$ tends to zero uniformly, as one can see from the following argument:
\begin{multline*}
\norm{u^N_{\eps_N}-u^N_s}_\infty=\sup_x\abs{\int_{\Pi^{d}} \theta^{\eps_{N}}(x-x')u^{N}(x')dx'-u^{N}(x)}=\\
\sup_x\abs{\int_{\Pi^{d}} \theta^{\eps_{N}}(x-x')\left(u^{N}(x')-u^{N}(x)\right)dx'}\\
\leq C\sup_x\int_{\Pi^{d}} \theta^{\eps_{N}}(x-x')\abs{x-x'}dx'\to 0.
\end{multline*}
Then the first term of the sum, $\EE{\Phi_{\phi}(u^{N},S^{N})}$, tends to zero. Regarding the second term, recalling the identity satisfied by the empirical measure, see $(\ref{eq::empiricalmeasure})$:
\begin{multline*}
\Psi_\psi(u^N,S^N)=\langle S^N_0-F_0,\psi_0\rangle+\int_0^T \langle g(u^{N}_{\eps}(x),v)\cdot \nabla_{v}\psi , S^{N}_{s}\rangle ds-\\
\int_0^T \langle g(u^{N}_s(x),v)\cdot \nabla _v \psi ,S^{N}_s\rangle ds+
+M^{N,\psi}_T.
\end{multline*}
Then
\begin{multline*}
\EE{\left|\Psi_\psi(u^N,S^N)\right|}\leq \\
\leq \EE{\left|\int_0^T \langle g(u^{N}_{\eps}(x),v)\cdot \nabla_{v}\psi , S^{N}_{s}\rangle ds- \int_0^T \langle g(u^{N}_s(x),v)\cdot \nabla _v \psi,S^{N}_s\rangle ds\right|}+\\
+\EE{ \langle S^N_0-F_0,\psi_0\rangle}+\EE{\abs{M^{N,\psi}_T}}. 
\end{multline*}
By the assumptions on the initial conditions $F^N_0$ and by standard argument the second and the third term are infinitesimal. Regarding the first term:
Following exactly the previous argument, it tends to zero.
\end{proof}

\noindent Combining Theorem \ref{teo::tightuN}, \ref{teo::tightSN} and \ref{teo::limit} we have proved Theorem \ref{teo::main}.

\end{document}